\documentclass[12pt]{article}
\usepackage{graphicx}

\usepackage{times}

\usepackage{amsthm, amsmath, amssymb}
\usepackage{microtype}
\usepackage{enumerate}
\usepackage{hyperref}
\usepackage{todonotes}
\makeatletter

\newtheorem{thm}{Theorem}
\newtheorem{lemma}[thm]{Lemma}

\theoremstyle{definition}

\DeclareMathOperator{\rank}{rank}

\newsavebox{\fmbox}

\newcommand{\reals}{\ensuremath{\mathbb{R}}}

\newlength{\dyindent}
{\end{list}}

\newenvironment{dy*}{\refstepcounter{equation}\begin{list}{}%
{\setlength{\leftmargin}{\dyindent}\setlength{\labelwidth}{\dyindent}%
\addtolength{\labelwidth}{-\labelsep}}%
\item}%
{\end{list}}

\newcommand{\mr}{\mbox{mr}}

\newcommand{\nullity}{\mbox{nullity }}

\newcommand{\sgn}{\mbox{sgn}}
\newtheorem{theorem}{Theorem}

\newtheorem{remark}[theorem]{Remark}

\title{The minimum rank of a sign pattern matrix with a $1$-separation}
\author{Marina Arav\\
 Frank J. Hall\\
 Zhongshan Li\\ 
 Hein van der Holst\footnote{Corresponding author, E-mail: hvanderholst@gsu.edu}\\ 
Lihua Zhang\\ 
Wenyan Zhou \\
Department of Mathematics and Statistics \\
Georgia State University \\
Atlanta, GA 30303, USA
}
\date{}
\begin{document}
\maketitle

\begin{abstract}
A sign pattern matrix is a matrix whose entries are from the set $\{+,-,0\}$. If $A$ is an $m\times n$ sign pattern matrix, the qualitative class of $A$, denoted $Q(A)$, is the set of all real $m\times n$ matrices $B=[b_{i,j}]$ with $b_{i,j}$ positive (respectively, negative, zero) if $a_{i,j}$ is $+$ (respectively, $-$, $0$). The minimum rank of a sign pattern matrix $A$, denoted $\mr(A)$, is the minimum of the ranks of the real matrices in $Q(A)$. Determination of the minimum rank of a sign pattern matrix is a longstanding open problem. 

For the case that the sign pattern matrix has a $1$-separation, we present a formula to compute the minimum rank of a sign pattern matrix using the minimum ranks of certain generalized sign pattern matrices associated with the $1$-separation.
\end{abstract}

\section*{Introduction}
A \emph{sign pattern matrix} (or \emph{sign pattern}) is a matrix whose entries are from the set $\{+,-,0\}$. If $B=[b_{i,j}]$ is a real matrix, then $\sgn(B)$ is the sign pattern matrix $A=[a_{i,j}]$ with $a_{i,j}=+$ (respectively, $-$, $0$) if $b_{i,j}$ is positive (respectively, negative, zero). If $A$ is a sign pattern matrix, the \emph{sign pattern class} of $A$, denoted $Q(A)$, is the set of all real matrices $B=[b_{i,j}]$ with $\sgn(B)=A$.
The \emph{minimum rank} of a sign pattern matrix $A$, denoted $\mr(A)$, is the minimum of the ranks of matrices in $Q(A)$; see \cite{Hall07}. Recently, Li et al. \cite{LivdHolst2013} obtained a characterization of sign pattern matrices $A$ with $\mr(A)\leq 2$. 
In this paper, we present a formula to compute the minimum rank of a sign pattern matrix with a $1$-separation using the minimum ranks of certain generalized sign pattern matrices associated with the $1$-separation.


The notion of sign pattern matrix can be extended to generalized sign pattern matrices by allowing certain entries to be $\#$; see \cite{Hall07}. For a generalized sign pattern matrix $A$, the generalized sign pattern class of $A$, denoted $Q(A)$, is defined by allowing  entries of a matrix $B=[b_{i,j}]\in Q(A)$ to be any real number if the corresponding entries of $A$ are $\#$. The minimum rank $\mr(A)$ of a generalized sign pattern matrix $A$ is defined in the same way as for a sign pattern matrix: $\mr(A)$ is the minimum of the ranks of matrices in $Q(A)$.
If $A=[a_{i,j}]$ and $C=[c_{i,j}]$ are generalized sign pattern matrices of the same size, we write $A\leq C$ if for each entry of $A$, $a_{i,j} = c_{i,j}$ or $c_{i,j} = \#$. It is clear that if $A\leq C$, then $Q(A)\subseteq Q(C)$. For a generalized sign pattern $C$, let $\mathcal{C}$ be the set of all sign pattern matrices $A$ such that $A\leq C$. Then, clearly, $Q(C) = \cup_{A\in \mathcal{C}} Q(A)$. Hence the minimum rank of a generalized sign pattern matrix $C$ equals $\min_{A\in \mathcal{C}} \mr(A)$.

We define subtraction of two elements from $\{+,-,0\}$ as follows:
\begin{enumerate}
\item $(+)-(0)=+,(0)-(-)=+, (+)-(-)=+$,
\item $(-)-(+)=-,(0)-(+)=-,(-)-(0)=-$,
\item $(0)-(0)=0$,
\item $(+)-(+)=\#,(-)-(-)=\#$.
\end{enumerate}
The idea behind the definition of, for example, $(-)-(+)=-$ is that subtracting a positive number from a negative number gives a negative number.


Let 
\begin{equation*}
M = \begin{bmatrix}
A_{1,1} & A_{1,2} & 0\\
A_{2,1} & a_{2,2} & A_{2,3}\\
0 & A_{3,2} & A_{3,3}
\end{bmatrix}
\end{equation*}
be a sign pattern matrix, where $A_{1,2}$ has only one column and $A_{2,1}$ only one row. We also say that the sign pattern matrix $M$ has a \emph{$1$-separation}.
For $p\in \{+,-,0\}$, let
\begin{equation*}
R_p = \begin{bmatrix}
A_{1,1} & A_{1,2}\\
A_{2,1} & p
\end{bmatrix}
\text{ and }
S_p = \begin{bmatrix}
a_{2,2}-p & A_{2,3}\\
A_{3,2} & A_{3,3}
\end{bmatrix}.
\end{equation*}
(Indeed $S_p$ might be a generalized sign pattern matrix.)
In this paper, we prove that the following formula holds:
\begin{equation}\label{mainformula}
\begin{split}
\mr(M)  = \min \{ & \mr(A_{1,1})+ \mr(A_{3,3})+2,\\
& \mr([A_{1,1}\;A_{1,2}])+\mr([A_{3,2}\;A_{3,3}])+1,\\
& \mr(\begin{bmatrix}
A_{1,1}\\
A_{2,1}
\end{bmatrix})+\mr(\begin{bmatrix}
A_{2,3}\\
A_{3,3}
\end{bmatrix})+1,\\
& \mr(R_+)+\mr(S_+),\\
& \mr(R_0)+\mr(S_0),\\
& \mr(R_-)+\mr(S_-)
\}
\end{split} 
\end{equation}
In the next section, we show that each of the terms in the minimum is at least $\mr(M)$. In Section~\ref{sec:minrank1sep}, we show that at least one of the terms in the minimum attains $\mr(M)$.

Formula~(\ref{mainformula}) is analogous to the formula for the minimum rank of $1$-sums of graphs. This formula was given by Hsieh \cite{Hsieh2001}, and, independently, by Barioli, Fallat, and Hogben \cite{BarFalHog2004a}. In case the graph is permitted to have loops, a formula was given by Mikkelson~\cite{Mikkelson:2008aa}.

The reader can see Fallat and Hogben \cite{FalHog2007} for a survey on the minimum ranks of graphs.

\section{Inequalities}

Let $0\leq k\leq m,n,r,s$ and let 
\begin{equation*}
A=\begin{bmatrix}
A_{1,1}&A_{1,2}\\
A_{2,1}&A_{2,2}
\end{bmatrix} \text{ and } B=\begin{bmatrix}
B_{1,1}&B_{1,2}\\
B_{2,1}&B_{2,2}
\end{bmatrix}
\end{equation*}
be $m\times n$ and $r\times s$ matrices with real entries, respectively, where $A_{2,2}$ and $B_{1,1}$ are $k\times k$. In \cite{FalJoh1999a}, the \emph{$k-$subdirect sum} of $A$ and $B$, denoted by $A\oplus_k B$, was introduced; this is the matrix 
\begin{equation*}A \oplus_{k} B=
\begin{bmatrix}
A_{1,1}&A_{1,2}&0\\
A_{2,1}&A_{2,2}+B_{1,1} & B_{1,2}\\
0 & B_{2,1} & B_{2,2}
\end{bmatrix}
\end{equation*}

\begin{lemma}\cite{MR3010007}\label{lem:sum}
Let 
\begin{equation*}
C = \begin{bmatrix}
C_{1,1} & C_{1,2}\\
C_{2,1} & C_{2,2}
\end{bmatrix}
\quad\text{and}\quad
D = \begin{bmatrix}
D_{1,1} & D_{1,2}\\
D_{2,1} & D_{2,2}
\end{bmatrix},
\end{equation*}
where $C_{2,2}$ and $D_{1,1}$ are $k\times k$ matrices.
Then $\rank(C\oplus_k D)\leq \rank(C\oplus D)$. 
\end{lemma}

In the next theorem, we show that each term in the minimum of Formula~(\ref{mainformula}) is at least $\mr(M)$.

\begin{thm}
Let 
\begin{equation*}
M=\begin{bmatrix}
A_{1,1}&A_{1,2}&0\\
A_{2,1}& a_{2,2} & A_{2,3}\\
0 & A_{3,2} & A_{3,3}\\
\end{bmatrix}
\end{equation*} 
be a sign pattern matrix, where $A_{1,1}$ is $m_1\times n_1$, $A_{1,2}$ is $m_1\times 1$, $A_{2,1}$ is $1\times n_1$, $a_{2,2}$ is $1\times 1$, $A_{2,3}$ is $1\times n_2$, $A_{3, 2}$ is $m_2\times 1$ and $A_{3,3}$ is $m_2\times n_2$. Then each of the following inequalities hold:
\begin{enumerate}[(i)]
\item $\mr(A_{1,1})+\mr(A_{3,3})+2\geq \mr(M)$,
\item $\mr([A_{1,1}\;A_{1,2}])+\mr([A_{3,2}\;A_{3,3}])+1\geq \mr(M)$,
\item $\mr(\begin{bmatrix}
A_{1,1}\\
A_{2,1}\\
\end{bmatrix})+\mr(\begin{bmatrix}
A_{2,3}\\
A_{3,3}\\
\end{bmatrix})+1\geq \mr(M)$, and
\item 
for each $p\in \{+,-,0\}$, 
\begin{equation*}
\mr(\begin{bmatrix}
A_{1,1} & A_{1,2}\\
A_{2,1} & p
\end{bmatrix}
)+\mr(
\begin{bmatrix}
a_{2,2}-p & A_{2,3}\\
A_{3,2} & A_{3,3}
\end{bmatrix}
)\geq \mr(M).
\end{equation*} 
\end{enumerate}
\end{thm}

\begin{proof}
To see that $\mr(A_{1,1})+\mr(A_{3,3})+2\geq \mr(M)$, let $C_{1,1}\in Q(A_{1,1})$ and $C_{3,3} \in Q(A_{3,3})$ such that $\rank(C_{1,1}) = \mr(A_{1,1})$ and $\rank(C_{3,3})=\mr(A_{3,3})$.
Let $C_{1,2}\in Q(A_{1,2})$, $C_{2,1}\in Q(A_{2,1})$, $C_{2,3} \in Q(A_{2,3})$, $C_{3,2}\in Q(A_{3,2})$, and $\sgn(c_{2,2}) = a_{2,2}$. Then
\begin{equation*}
\begin{split}
\mr(A_{1,1}) + \mr(A_{3,3})+2 & = \rank(\begin{bmatrix}
C_{1,1} & 0\\
0 & C_{3,3}
\end{bmatrix}) + 2\\
&\geq \rank(\begin{bmatrix}
C_{1,1} & C_{1,2} & 0\\
C_{2,1} & c_{2,2} & C_{2,3}\\
0 & C_{3,2} & C_{3,3}
\end{bmatrix})\\
&\geq \mr(M),
\end{split}
\end{equation*}

To see that $\mr([A_{1,1}\;A_{1,2}])+\mr([A_{3,2}\;A_{3,3}])+1\geq \mr(M)$, let $[C_{1,1}\;C_{1,2}]\in Q([A_{1,1}\;A_{1,2}])$ and $[C_{3,2}\;C_{3,3}]\in Q([A_{3,2}\;B_{3,3}])$ be such that $\rank([C_{1,1}\;C_{1,2}])=\mr([A_{1,1}\;A_{1,2}])$ and $\rank([C_{3,2}\;C_{3,3}])=\mr([A_{3,2}\;A_{3,3}])$. Clearly,
\begin{equation*}
\rank([C_{1,1}\;C_{1,2}]) + \rank([C_{3,2}\;C_{3,3}])\geq \rank(\begin{bmatrix}
C_{1,1} & C_{1,2}  & 0\\
0 & C_{3,2} & C_{3,3}
\end{bmatrix}).
\end{equation*}
Since
\begin{equation*}
\rank(\begin{bmatrix}
C_{1,1} & C_{1,2}  & 0\\
0 & C_{3,2} & C_{3,3}
\end{bmatrix}) + 1\geq \rank(
\begin{bmatrix}
C_{1,1} & C_{1,2} & 0\\
C_{2,1} & c_{2,2} & C_{2,3}\\
0 & C_{3,2} & C_{3,3}
\end{bmatrix}),
\end{equation*}
we obtain $\mr([A_{1,1}\;A_{1,2}])+\mr([A_{3,2}\;A_{3,3}])+1\geq \mr(M)$.

The proof that 
$\mr(\begin{bmatrix}
A_{1,1}\\
A_{2,1}\\
\end{bmatrix})+\mr(\begin{bmatrix}
B_{1,2}\\
B_{2,2}\\
\end{bmatrix})+1\geq \mr(M)$ is similar to the proof of the previous case.

Let $p\in \{+,-,0\}$. To shorten notation, let
\begin{equation*}
R_p = \begin{bmatrix}
A_{1,1} & A_{1,2}\\
A_{2,1} & p
\end{bmatrix}
\text{ and }
S_p = \begin{bmatrix}
a_{2,2}-p & A_{2,3}\\
A_{3,2} & A_{3,3}
\end{bmatrix}.
\end{equation*}
To see that $\mr(R_p) + \mr(S_p) \geq \mr(M)$,
let
\begin{equation*}
C = \begin{bmatrix}
C_{1,1} & C_{1,2}\\
C_{2,1} & c
\end{bmatrix}\in Q(R_p)
\text{ and }
D = \begin{bmatrix}
d & C_{2,3}\\
C_{3,2} & C_{3,3}
\end{bmatrix}\in Q(S_p)
\end{equation*}
be such that $\rank(C) = \mr(R_p)$ and $\rank(D)=\mr(S_p)$.

We now do a case-checking.

Suppose first that $a_{2,2}-p=0$. Then $p=0$ and $a_{2,2}=0$.
Hence $c=0$ and $d=0$. Then $C\oplus_1 D \in Q(M)$, and, by Lemma~\ref{lem:sum}, $\mr(M)\leq \rank(C\oplus_1 D)\leq \rank(C)+\rank(D) = \mr(R_p) + \mr(S_p)$. 

Suppose next that $a_{2,2}-p=+$. Then one of the following holds:
\begin{enumerate}
\item $p=-$ and $a_{2,2}=0$, 
\item $p=0$ and $a_{2,2}=+$, and
\item $p=-$ and $a_{2,2}=+$.
\end{enumerate}
Suppose $p=-$ and $a_{2,2}=0$. By scaling $D$ by a positive scalar, we may assume that $d=-c$. Then $C\oplus_1 D\in Q(M)$, and, by Lemma~\ref{lem:sum}, $\mr(M) \leq \rank(C\oplus_1 D)\leq \rank(C)+\rank(D) = \mr(R_p) + \mr(S_p)$. Suppose $p=0$ and $a_{2,2}=+$. Then $C\oplus_1 D\in Q(M)$, and, by Lemma~\ref{lem:sum}, $\mr(M) \leq \rank(C\oplus_1 D)\leq \rank(C)+\rank(D) = \mr(R_p) + \mr(S_p)$. Suppose $p=-$ and $a_{2,2}=+$. By scaling $D$ by a positive scalar, we may assume that $c+d > 0$. Then $C\oplus_1 D\in Q(M)$, and, by Lemma~\ref{lem:sum}, $\mr(M)\leq \rank(C\oplus_1 D)\leq \rank(C)+\rank(D) = \mr(R_p) + \mr(S_p)$. 

The case where $a_{2,2}-p=-$ is similar.

Suppose finally that $a_{2,2}-p=\#$. Then one of the following holds:
\begin{enumerate}
\item $p=+$ and $a_{2,2}=+$, and
\item $p=-$ and $a_{2,2}=-$.
\end{enumerate}
Suppose $p=+$ and $a_{2,2}=+$. Then $C\oplus_1 D\in Q(M)$, and, by Lemma~\ref{lem:sum}, $\mr(M) \leq \rank(C\oplus_1 D)\leq \rank(C)+\rank(D) = \mr(R_p) + \mr(S_p)$. The case where $p=-$ and $a_{2,2}=-$ is similar.
\end{proof}

\section{Minimum Rank of Sign Pattern with 1-Separation}\label{sec:minrank1sep}

In this section we finish the proof that Formula~(\ref{mainformula}) is correct. First we prove some lemmas.

\begin{lemma}\label{lem:vertexadd}
For any $m\times n$ real matrix $B$ with $m, n\geq 1$, and any nonzero real numbers $a$ and $c$,
\begin{equation*}
\rank(\begin{bmatrix}
0 & a & 0\\
c & b_{1,1} & B_{1,2}\\
0 & B_{2,1} & B_{2,2}
\end{bmatrix}) = \rank(B_{2,2}) + 2.
\end{equation*}
\end{lemma}
\begin{proof}
Let
\begin{equation*}
P = \begin{bmatrix}
\frac{1}{a} & 0 & 0\\
-\frac{b_{1,1}}{2a} & 1 & 0\\
-\frac{B_{2,1}}{a} & 0 & I_{m-2}
\end{bmatrix}\quad\text{and}\quad
Q = \begin{bmatrix}
\frac{1}{c} & -\frac{b_{1,1}}{2c} & -\frac{B_{1,2}}{c}\\
0 & 1 & 0\\
0 & 0 & I_{n-2}
\end{bmatrix}.
\end{equation*}
Then
\begin{equation*}
P \begin{bmatrix}
0 & a & 0\\
c & b_{1,1} & B_{1,2}\\
0 & B_{2,1} & B_{2,2}
\end{bmatrix} Q = 
\begin{bmatrix}
0 & 1 & 0\\
1 & 0 & 0\\
0 & 0 & B_{2,2}
\end{bmatrix}.
\end{equation*}
From this the lemma easily follows.
\end{proof}

\begin{lemma}\label{lem:adjoin}
Let $A = \begin{bmatrix}
A_{1,1} & A_{1,2}\\
A_{2,1} & A_{2,2}
\end{bmatrix}$ be a real matrix, where $A_{1,1}$ is $m_1\times n_1$, $A_{1,2}$ is $m_1\times n_2$, $A_{2,1}$ is $m_2\times n_1$, and $A_{2,2}$ is $m_2\times n_2$.
If $x\in \ker(A_{2,2}^T)$ and $y\in \ker(A_{2,2})$, then
\begin{equation*}
\rank \begin{bmatrix}
0 & x^T A_{2,1} & 0\\
A_{1,2} y & A_{1,1} & A_{1,2}\\
0 & A_{2,1} & A_{2,2}
\end{bmatrix} = \rank A.
\end{equation*}
\end{lemma}
\begin{proof}
Let
\begin{equation*}
P = \begin{bmatrix}
0 & x^T\\
I_{m_1} & 0\\
0 & I_{m_2}
\end{bmatrix}\quad\text{and}\quad
Q= \begin{bmatrix}
0 & I_{n_1} & 0\\
y & 0 & I_{n_2}
\end{bmatrix}.
\end{equation*}
Then 
\begin{equation*}
P A Q= \begin{bmatrix}
0 & x^T A_{2,1} & 0\\
A_{1,2} y & A_{1,1} & A_{1,2}\\
0 & A_{2,1} & A_{2,2}
\end{bmatrix}.
\end{equation*}
Hence, $\rank \begin{bmatrix}
0 & x^T A_{2,1} & 0\\
A_{1,2} y & A_{1,1} & A_{1,2}\\
0 & A_{2,1} & A_{2,2}
\end{bmatrix} \leq \rank A$. The other inequality is clear.
\end{proof}

\begin{lemma}\label{lem:decomp}
Let $A = \begin{bmatrix}
A_{1,1} & A_{1,2} & 0\\
A_{2,1} & a_{2,2} & A_{2,3}\\
0 & A_{3,2} & A_{3,3}
\end{bmatrix}$ be an $m\times n$ real matrix, where $A_{1,1}$ is $m_1\times n_1$ and $A_{3,3}$ is $m_2\times n_2$, (and so $m=m_1+m_2+1$ and $n=n_1+n_2+1$). Then at least one of the following holds:
\begin{enumerate}[(i)]
\item There exist vectors $v \in \mathbb{R}^{m_1}$ and $z\in \mathbb{R}^{n_1}$ such that
\begin{equation*}
\rank(\begin{bmatrix}
A_{1,1} & A_{1,2}\\
A_{2,1} & v^T A_{1,1} z
\end{bmatrix})+\rank(
\begin{bmatrix}
a_{2,2} - v^T A_{1,1} z & A_{2,3}\\
A_{3,2} & A_{3,3}
\end{bmatrix}) = \rank(A).
\end{equation*}
\item\label{item:sumcase2} $\rank(\begin{bmatrix}
A_{1,1}\\
A_{2,1}
\end{bmatrix}) + \rank(\begin{bmatrix}
A_{2,3}\\
A_{3,3}
\end{bmatrix})+1 = \rank(A)$.
\item\label{item:sumcase3} $\rank([A_{1,1}\;A_{1,2}]) + \rank([A_{3,2}\;A_{3,3}]) +1 = \rank(A)$.
\item $\rank(A_{1,1}) + \rank(A_{3,3}) + 2=\rank(A)$.
\end{enumerate}
\end{lemma}

\begin{proof}
Suppose first that $[A_{2,1}\;A_{2,3}]x = 0$ for all $x\in \ker(A_{1,1}\oplus A_{3,3})$ and that $y^T \begin{bmatrix}
A_{1,2}\\
A_{3,2}
\end{bmatrix} = 0$ for all $y\in \ker((A_{1,1}\oplus A_{3,3})^T)$.
Then there exist a vector $v\in \mathbb{R}^{m_1}$ such that $v^T A_{1,1} = A_{2,1}$ and a vector $z\in \mathbb{R}^{n_1}$ such that $A_{1,1} z = A_{1,2}$. Let
\begin{equation*}
P = \begin{bmatrix}
I_{m_1} & 0 & 0\\
v^T & 0 & 0\\
-v^T & 1 & 0\\
0 & 0 & I_{m_2}
\end{bmatrix}\quad\text{and}\quad
Q=\begin{bmatrix}
I_{n_1} & z & -z & 0\\
0 & 0 & 1 & 0\\
0 & 0 & 0 & I_{n_2}
\end{bmatrix}.
\end{equation*}
A calculation shows that 
\begin{equation*}
P A Q = \begin{bmatrix}
A_{1,1} & A_{1,2}\\
A_{2,1} & v^T A_{1,1} z
\end{bmatrix}\oplus
\begin{bmatrix}
a_{2,2} - v^T A_{1,1} z & A_{2,3}\\
A_{3,2} & A_{3,3}
\end{bmatrix}.
\end{equation*}
Hence 
\begin{equation*}
\rank(A) \geq \rank(\begin{bmatrix}
A_{1,1} & A_{1,2}\\
A_{2,1} & v^T A_{1,1} z
\end{bmatrix}) + \rank(
\begin{bmatrix}
a_{2,2} - v^T A_{1,1} z & A_{2,3}\\
A_{3,2} & A_{3,3}
\end{bmatrix})
\end{equation*}
By Lemma~\ref{lem:sum}, the opposite inequality also holds.

Suppose next that $[A_{2,1}\;A_{2,3}]x = 0$ for all $x\in \ker(A_{1,1}\oplus A_{3,3})$ and that there exists a  $y\in \ker((A_{1,1}\oplus A_{3,3})^T)$ such that $y^T \begin{bmatrix}
A_{1,2}\\
A_{3,2}
\end{bmatrix} = e \not= 0$.
By Lemma~\ref{lem:adjoin}, 
\begin{equation*}
\rank(\begin{bmatrix}
0 & 0 & e & 0\\
0 & A_{1,1} & A_{1,2} & 0\\
0 & A_{2,1} & a_{2,2} & A_{2,3}\\
0 & 0 & A_{3,2} & A_{3,3}
\end{bmatrix}) = \rank(A).
\end{equation*}
Hence 
\begin{equation*}
1 + \rank(
\begin{bmatrix}
A_{1,1} & 0\\
A_{2,1} & A_{2,3}\\
0 & A_{3,3}
\end{bmatrix}) = \rank(A).
\end{equation*}
Since 
\begin{equation*}
\nullity(\begin{bmatrix}
A_{1,1} & 0\\
A_{2,1} & A_{2,3}\\
0 & A_{3,3}
\end{bmatrix}) = 
\nullity(
\begin{bmatrix}
A_{1,1} & 0\\
0 & A_{3,3}
\end{bmatrix}),
\end{equation*}
we obtain 
\begin{equation*}
\rank(\begin{bmatrix}
A_{1,1} & 0\\
A_{2,1} & A_{2,3}\\
0 & A_{3,3}
\end{bmatrix}) = 
\rank(\begin{bmatrix}
A_{1,1} & 0\\
0 & A_{3,3}
\end{bmatrix}).
\end{equation*}
Hence $\rank(A) = \rank(A_{1,1}) + \rank(A_{3,3})+1$. 
From $[A_{2,1}\;A_{2,3}]x = 0$ for all $x\in \ker(A_{1,1}\oplus A_{3,3})$, it follows that $\rank(\begin{bmatrix}
A_{1,1}\\
A_{1,2}
\end{bmatrix})
 = \rank(A_{1,1})$ and $\rank(
\begin{bmatrix} 
 A_{2,3}\\
 A_{3,3}
\end{bmatrix}) = \rank(A_{3,3})$. Thus $\rank(\begin{bmatrix}
A_{1,1}\\
A_{2,1}
\end{bmatrix}) + \rank(\begin{bmatrix}
A_{2,3}\\
A_{3,3}
\end{bmatrix})+1 = \rank(A)$.

The case that there exists an $x\in \ker(A_{1,1}\oplus A_{3,3})$ such that  $[A_{2,1}\;A_{2,3}]x$ is nonzero and $y^T \begin{bmatrix}
A_{1,2}\\
A_{3,2}
\end{bmatrix} = 0$ for all $y\in \ker((A_{1,1}\oplus A_{3,3})^T)$ yields $\rank([A_{1,1}\;A_{1,2}]) + \rank([A_{3,2}\;A_{3,3}]) +1 = \rank(A)$.

Hence, we are left with the case that there exist an $x\in \ker(A_{1,1}\oplus A_{3,3})$ such that  $f = [A_{2,1}\;A_{2,3}]x$ is nonzero and there exists a $y\in \ker((A_{1,1}\oplus A_{3,3})^T)$ such that $e = y^T \begin{bmatrix}
A_{1,2}\\
A_{3,2}
\end{bmatrix}$ is nonzero.
Then, by Lemma~\ref{lem:adjoin}, 
\begin{equation*}
\rank(\begin{bmatrix}
0 & 0 & e & 0\\
0 & A_{1,1} & A_{1,2} & 0\\
f & A_{2,1} & a_{2,2} & A_{2,3}\\
0 & 0 & A_{3,2} & A_{3,3}
\end{bmatrix}) = \rank(A).
\end{equation*}
By Lemma~\ref{lem:vertexadd}, 
\begin{equation*}
\rank(\begin{bmatrix}
A_{1,1} & 0\\
0 & A_{3,3}
\end{bmatrix}) + 2 = \rank(A).
\end{equation*}
Thus $\rank(A_{1,1}) + \rank(A_{3,3}) + 2 = \rank(A)$.
\end{proof}

\begin{remark}
The proof shows that if Case~(\ref{item:sumcase2}) happens, then $A_{2,1}$ belongs to the row space of $A_{1,1}$, and $A_{2,3}$ belongs to the row space of $A_{3,3}$. Similarly, if Case~(\ref{item:sumcase3}) happens, then $A_{1,2}$ belongs to the column space of $A_{1,1}$, and $A_{2,3}$ belongs to the column space of $A_{3,3}$.
\end{remark}

\begin{thm}
Let
\begin{equation*}
M = \begin{bmatrix}
A_{1,1} & A_{1,2}& 0\\
A_{2,1} & a_{2,2} & A_{2,3}\\
0 & A_{3,2} & A_{3,3}
\end{bmatrix},
\end{equation*}
where $A_{1,2}$ is $n_1\times 1$, $A_{2,1}$ is $m_1\times 1$, and, for $p\in \{+,-,0\}$, let
\begin{equation*}
R_p = \begin{bmatrix}
A_{1,1} & A_{1,2}\\
A_{2,1} & p
\end{bmatrix}
\text{ and }
S = \begin{bmatrix}
a_{2,2}-p & A_{2,3}\\
A_{3,2} & A_{3,3}
\end{bmatrix}.
\end{equation*}
Then 
\begin{equation*}
\begin{split}
\mr(M)  = \min \{ & \mr(A_{1,1})+ \mr(A_{3,3})+2,\\
& \mr([A_{1,1}\;A_{1,2}])+\mr([A_{3,2}\;A_{3,3}])+1,\\
& \mr(\begin{bmatrix}
A_{1,1}\\
A_{2,1}
\end{bmatrix})+\mr(\begin{bmatrix}
A_{2,3}\\
A_{3,3}
\end{bmatrix})+1,\\
& \mr(R_+)+\mr(S_+),\\
& \mr(R_0)+\mr(S_0),\\
& \mr(R_-)+\mr(S_-)
\}
\end{split} 
\end{equation*}
\end{thm}

\begin{proof}
By the previous section,
\begin{equation}\label{eq:formula2}
\begin{split}
\mr(M) \leq \min \{ & \mr(A_{1,1})+ \mr(A_{3,3})+2,\\
& \mr([A_{1,1}\;A_{1,2}])+\mr([A_{3,2}\;A_{3,3}])+1,\\
& \mr(\begin{bmatrix}
A_{1,1}\\
A_{2,1}
\end{bmatrix})+\mr(\begin{bmatrix}
A_{2,3}\\
A_{3,3}
\end{bmatrix})+1,\\
& \mr(R_+)+\mr(S_+),\\
& \mr(R_0)+\mr(S_0),\\
& \mr(R_-)+\mr(S_-)
\} 
\end{split} 
\end{equation}
We now show that at least one of the terms in the minimum on the right-hand side of $(\ref{eq:formula2})$ equals $\mr(M)$.

Let
\begin{equation*}
C = \begin{bmatrix}
C_{1,1} & C_{1,2} & 0\\
C_{2,1} & c_{2,2} & C_{2,3}\\
0 & C_{3,2} & C_{3,3}
\end{bmatrix}\in Q(M)
\end{equation*}
be such that $\rank(C) = \mr(M)$.
Then, by Lemma~\ref{lem:decomp}, at least one of the following holds:
\begin{enumerate}[(i)]
\item\label{item1} There exist vectors $v\in\reals^{n_1}$ and $z\in\reals^{m_1}$ such that
\begin{equation*}
\rank(\begin{bmatrix}
C_{1,1} & C_{1,2}\\
C_{2,1} & v^T C_{1,1} z
\end{bmatrix})+\rank(
\begin{bmatrix}
c_{2,2} - v^T C_{1,1} z & C_{2,3}\\
C_{3,2} & C_{3,3}
\end{bmatrix}) = \rank(C).
\end{equation*}
\item\label{item2} $\rank(\begin{bmatrix}
C_{1,1}\\
C_{2,1}
\end{bmatrix}) + \rank(\begin{bmatrix}
C_{2,3}\\
C_{3,3}
\end{bmatrix})+1 = \rank(C)$.
\item\label{item3} $\rank([C_{1,1}\;C_{1,2}]) + \rank([C_{3,2}\;C_{3,3}]) +1 = \rank(C)$.
\item\label{item4} $\rank(C_{1,1}) + \rank(C_{3,3}) + 2=\rank(C)$.
\end{enumerate}

Suppose first that $(\ref{item2})$ holds. Then 
\begin{equation*}
\begin{split}
\mr(\begin{bmatrix}
A_{1,1}\\
A_{2,1}
\end{bmatrix}) + \mr(\begin{bmatrix}
A_{2,3}\\
A_{3,3}
\end{bmatrix})+1 & \leq  \rank(\begin{bmatrix}
C_{1,1}\\
C_{2,1}
\end{bmatrix}) + \rank(\begin{bmatrix}
C_{2,3}\\
C_{3,3}
\end{bmatrix})+1\\
& = \rank(C) = \mr(M).
\end{split}
\end{equation*} 

Case $(\ref{item3})$ is similar to $(\ref{item2})$. 

Suppose next that $(\ref{item4})$ holds. Then
\begin{equation*}
\begin{split}
\mr(A_{1,1}) + \mr(A_{3,3}) + 2 & \leq \rank(C_{1,1}) + \rank(C_{3,3}) + 2 \\
& = \rank(C) = \mr(M).
\end{split}
\end{equation*}

Suppose finally that $(\ref{item1})$ holds. If $v^T C_{1,1} z > 0$, then 
\begin{equation*}
\begin{bmatrix}
C_{1,1} & C_{1,2}\\
C_{2,1} & v^T C_{1,1} z
\end{bmatrix}\in Q(R_+)\quad\text{and}\quad
\begin{bmatrix}
c_{2,2} - v^T C_{1,1} z & C_{2,3}\\
C_{3,2} & C_{3,3}
\end{bmatrix}\in Q(S_+).
\end{equation*}
Hence 
\begin{equation*}
\begin{split}
\mr(M) = \rank(C) &= \rank(\begin{bmatrix}
C_{1,1} & C_{1,2}\\
C_{2,1} & v^T C_{1,1} z
\end{bmatrix})+\rank(
\begin{bmatrix}
c_{2,2} - v^T C_{1,1} z & C_{2,3}\\
C_{3,2} & C_{3,3}
\end{bmatrix}) \\
&\geq \mr(R_+) + \mr(S_+).
\end{split}
\end{equation*}
The cases where $v^T C_{1,1} z = 0$ and $v^T C_{1,1} z < 0$ are similar.
\end{proof}

\section{Examples}
We exhibit several examples of sign pattern matrices illustrating that each term in Formula~(\ref{mainformula}) is needed.

To see that the term $\mr(A_{1,1})+\mr(A_{3,3})+2$ is needed in Formula~(\ref{mainformula}), let 
\begin{equation*}
M = \begin{bmatrix}
A_{1,1} & A_{1,2}& 0\\
A_{2,1} & a_{2,2} & A_{2,3}\\
0 & A_{3,2} & A_{3,3}
\end{bmatrix}=\begin{bmatrix}
0 & + & 0\\
+ & 0 & +\\
0 & + & 0
\end{bmatrix},
\end{equation*} 
and, for $p\in \{+,-,0\}$, let
\begin{equation*}
R_p = \begin{bmatrix}
0 & +\\
+ & p
\end{bmatrix}
\text{ and }
S_p = \begin{bmatrix}
-p & +\\
+ & 0
\end{bmatrix}.
\end{equation*}
Observe that $\mr(M)=2$. Note that $\mr(A_{1,1})+ \mr(A_{3,3})+2=0+0+2=2$, while $\mr([A_{1,1}\;A_{1,2}])+\mr([A_{3,2}\;A_{3,3}])+1=1+1+1=3$, $\mr(\begin{bmatrix}
A_{1,1}\\
A_{2,1}
\end{bmatrix})+\mr(\begin{bmatrix}
A_{2,3}\\
A_{3,3}
\end{bmatrix})+1=1+1+1=3$, $\mr(R_+)+\mr(S_+)=2+2=4$, $\mr(R_0)+\mr(S_0)=2+2=4$, and $\mr(R_-)+\mr(S_-)=2+2=4$.

To see that the term
$\mr(\begin{bmatrix}
A_{1,1} & A_{1,2}
\end{bmatrix}) + 
\mr(\begin{bmatrix}
A_{3,2}& A_{3,3}
\end{bmatrix})+1$
is needed in Formula~(\ref{mainformula}), let 
\begin{equation*}
M = \begin{bmatrix}
A_{1,1} & A_{1,2} & 0\\
A_{2,1} & a_{2,2} & A_{2,3}\\
0 & A_{3,2} & A_{3,3}
\end{bmatrix} =
\begin{bmatrix}
+ & + & 0 & 0 & 0\\
+ & + & 0 & 0 & 0\\
0 & + & + & + & 0\\
0 & 0 & + & 0 & +\\
\end{bmatrix},
\end{equation*}
and let 
\begin{equation*}
R_p = \begin{bmatrix}
A_{1,1} & A_{1,2}\\
A_{2,1} & p
\end{bmatrix}=
\begin{bmatrix}
+ & + & 0\\
+ & + & 0\\
0 & + & p
\end{bmatrix}
\end{equation*}
and 
\begin{equation*}
S_p = \begin{bmatrix}
a_{2,2}-p & A_{2,3}\\
A_{3,2} & A_{3,3}
\end{bmatrix}=
\begin{bmatrix}
(+)-p & + & 0\\
+ & 0 & +
\end{bmatrix}.
\end{equation*}
Observe that $\mr(M) = 3$. Note that $\mr(A_{1,1}) + \mr(A_{3,3}) + 2 = 1+1+2 = 4$, $\mr([A_{1,1}\;A_{1,2}]) + \mr([A_{3,2}\;A_{3,3}]) + 1 = 1 + 1 + 1 = 3, \mr(\begin{bmatrix}
A_{1,1}\\
A_{2,1}
\end{bmatrix}) + 
\mr(\begin{bmatrix}
A_{2,3}\\
A_{3,3}
\end{bmatrix})+1 = 2 + 2 + 1 = 5$, $\mr(R_+) + \mr(S_+) = 2 + 2 = 4$, $\mr(R_0) + \mr(S_0) = 2 + 2 = 4$, $\mr(R_-) + \mr(S_-) = 2+2=4$.

Taking the transpose of $M$ in the previous example shows that the term 
$\mr(\begin{bmatrix}
A_{1,1}\\
A_{2,1}
\end{bmatrix}) + 
\mr(\begin{bmatrix}
A_{2,3}\\
A_{3,3}
\end{bmatrix})+1$
is needed in Formula~(\ref{mainformula}).

To see that the term $\mr(R_+)+\mr(S_+)$ is needed in Formula~(\ref{mainformula}), let
\begin{equation*} 
M = \begin{bmatrix}
A_{1,1} & A_{1,2}& 0\\
A_{2,1} & a_{2,2} & A_{2,3}\\
0 & A_{3,2} & A_{3,3}
\end{bmatrix}=\begin{bmatrix}
+ & + & 0\\
+ & - & -\\
0 & + & +
\end{bmatrix},
\end{equation*}
and let
\begin{equation*}
R_p = \begin{bmatrix}
A_{1,1} & A_{1,2}\\
A_{2,1} & p
\end{bmatrix}=
\begin{bmatrix}
+ & +\\
+ & p
\end{bmatrix}
\end{equation*}
and
\begin{equation*}
S_p = \begin{bmatrix}
a_{2,2}-p & A_{2,3}\\
A_{3,2} & A_{3,3}
\end{bmatrix}=
\begin{bmatrix}
(-)-p & -\\
+ & +
\end{bmatrix}.
\end{equation*}

Observe that $\mr(M)=2$. Note that $\mr(A_{1,1})+ \mr(A_{3,3})+2=1+1+2=4, \mr([A_{1,1}\;A_{1,2}])+\mr([A_{3,2}\;A_{3,3}])+1=1+1+1=3,\mr(\begin{bmatrix}
A_{1,1}\\
A_{2,1}
\end{bmatrix})+\mr(\begin{bmatrix}
A_{2,3}\\
A_{3,3}
\end{bmatrix})+1=1+1+1=3, \mr(R_+)+\mr(S_+)=1+1=2, \mr(R_0)+\mr(S_0)=2+1=3,\mr(R_-)+\mr(S_-)=2+1=3$. 

Taking $-M$ in the previous example shows that the term 
$\mr(R_-)+\mr(S_-)$ is needed in Formula~(\ref{mainformula}).

To see that the term $\mr(R_0)+\mr(S_0)$ is needed in Formula~(\ref{mainformula}), let
\begin{equation*}
M = \begin{bmatrix}
A_{1,1} & A_{1,2}& 0\\
A_{2,1} & a_{2,2} & A_{2,3}\\
0 & A_{3,2} & A_{3,3}
\end{bmatrix}=\begin{bmatrix}
+ & 0 & 0\\
+ & - & +\\
0 & + & -
\end{bmatrix},
\end{equation*} 
and let
\begin{equation*}
R_p = \begin{bmatrix}
A_{1,1} & A_{1,2}\\
A_{2,1} & p
\end{bmatrix}=
\begin{bmatrix}
+ & +\\
+ & p
\end{bmatrix}
\end{equation*}
and
\begin{equation*}
S_p = \begin{bmatrix}
a_{2,2}-p & A_{2,3}\\
A_{3,2} & A_{3,3}
\end{bmatrix}=
\begin{bmatrix}
(-)-p & -\\
+ & +
\end{bmatrix}.
\end{equation*}

Observe that $\mr(M)=2$. Note that $\mr(A_{1,1})+ \mr(A_{3,3})+2=1+1+2=4, \mr([A_{1,1}\;A_{1,2}])+\mr([A_{3,2}\;A_{3,3}])+1=1+1+1=3,\mr(\begin{bmatrix}
A_{1,1}\\
A_{2,1}
\end{bmatrix})+\mr(\begin{bmatrix}
A_{2,3}\\
A_{3,3}
\end{bmatrix})+1=1+1+1=3, \mr(R_+)+\mr(S_+)=2+1=3, \mr(R_0)+\mr(S_0)=1+1=2,\mr(R_-)+\mr(S_-)=2+1=3$.

\newcommand{\noopsort}[1]{}


\begin{thebibliography}{1}

\bibitem{MR3010007}
F.~Barioli, W.~Barrett, S.~M. Fallat, H.~T. Hall, L.~Hogben, B.~Shader,
  P.~van~den Driessche, and H.~van~der Holst.
\newblock Parameters related to tree-width, zero forcing, and maximum nullity
  of a graph.
\newblock {\em J. Graph Theory}, 72(2):146--177, 2013.

\bibitem{BarFalHog2004a}
F.~Barioli, S.~Fallat, and L.~Hogben.
\newblock Computation of minimal rank and path cover number for certain graphs.
\newblock {\em Linear Algebra Appl.}, 392:289--303, 2004.

\bibitem{FalHog2007}
S.~Fallat and L.~Hogben.
\newblock The minimum rank of symmetric matrices described by a graph: A
  survey.
\newblock {\em Linear Algebra Appl.}, 426(2--3):558--582, 2007.

\bibitem{FalJoh1999a}
S.~M. Fallat and C.~R. Johnson.
\newblock Sub-direct sums and positivity classes of matrices.
\newblock {\em Linear Algebra Appl.}, 288:149--173, February 1999.

\bibitem{Hall07}
F.~J. Hall and Z.~Li.
\newblock Sign pattern matrices.
\newblock In L.~Hogben, editor, {\em Handbook of Linear Algebra}, chapter~33.
  Simon and Hall/CRC Press, 2007.

\bibitem{Hsieh2001}
L.-Y. Hsieh.
\newblock {\em On Minimum Rank Matrices having a Prescribed Graph}.
\newblock PhD thesis, University of Wisconsin-Madison, 2001.

\bibitem{LivdHolst2013}
Z.~Li, Y.~Gao, M.~Arav, F.~Gong, W.~Gao, F.~Hall, and H.~van~der Holst.
\newblock Sign patterns with minimum rank 2 and upper bounds on minimum ranks.
\newblock {\em Linear and Multilinear Algebra}, 61:895--908, 2013.

\bibitem{Mikkelson:2008aa}
R.~C. Mikkelson.
\newblock {\em Minimum rank of graphs that allow loops}.
\newblock PhD thesis, Iowa State University, 2008.

\end{thebibliography}

\end{document}